\documentclass{article}

\usepackage[left=30truemm, right=30truemm]{geometry}
\usepackage{amsmath,amssymb,amsthm}
\usepackage{bm}
\usepackage{mathtools}
\usepackage{xcolor}
\usepackage{graphicx}
\usepackage{enumitem}[shortlabels]

\usepackage{xcolor,graphicx}
\usepackage{tikz}
\usetikzlibrary{calc}

\theoremstyle{definition}
\newtheorem{theorem}{Theorem}[]
\newtheorem{proposition}[theorem]{Proposition}
\newtheorem{corollary}[theorem]{Corollary}

\newtheorem{lemma}[theorem]{Lemma}
\newtheorem{claim}[]{Claim}

\newtheorem{conjecture}[theorem]{Conjecture}

\numberwithin{equation}{section}

\title{New type degree conditions for a graph to have a 2-factor}
\author{Masaki Kashima\thanks{School of Fundamental Science and Technology,
Graduate School of Science and Technology, Keio University, 3-14-1 Hiyoshi, Kohoku-ku, Yokohama 223-8522, Japan. email: masaki.kashima10@gmail.com}}

\begin{document}

\maketitle

\begin{abstract}
    A 2-factor of a graph is a 2-regular spanning subgraph.
    For a graph $G$ and an independent set $I$ of $G$, let $\delta_G(I)$ denote the minimum degree of vertices contained in $I$.
    We show that (1) if every independent set $I$ of $G$ satisfies $|I|\leq \delta_G(I)-1$, then $G$ has a 2-factor and that (2) if every independent set $I$ of $G$ satisfies $|I|\leq \delta_G(I)$, then $G$ has a 2-factor unless $G$ is isomorphic to a graph in completely determined exceptional graphs.
    It can be easily shown that the assumption of (1) is a relaxation of the Dirac condition on Hamiltonicity of graphs, and that the assumption of (2) is a relaxation of the Chv\'{a}tal-Erd\H{o}s condition on Hamiltonicity of graphs.
    Furthermore, for graphs with the assumption of (1), we show some results on a 2-factor with a bounded number of cycles.
\end{abstract}

\section{Introduction}

Throughout the paper, we only consider simple, finite, and undirected graphs.
For a positive integer $k$, let $[k]$ denote the set of positive integers at most $k$.
For a graph $G$, let $|G|$ denote the order of $G$.
Let $d_G(v)$ denote the degree of a vertex $v$ of $G$, and let $\delta(G)$ denote the minimum degree of $G$.

A \emph{2-factor} of a graph is a spanning 2-regular subgraph.
Motivated by studies on Hamilton cycles, which is a connected 2-factor, sufficient conditions for a graph to have a 2-factor have been actively studied over decades.

For a graph $G$ and a positive integer $k$, we define 
\[\sigma_k(G):=\min\left\{\sum_{v\in I}d_G(v)\;\middle|\; I\text{ is an independent set of order $k$ of $G$}\right\}\]
if $\alpha(G)\geq k$, and define $\sigma_k(G):=\infty$ otherwise, where $\alpha(G)$ is the independence number of $G$.
Note that $\sigma_1(G)=\delta(G)$.
Ore's well-known result on Hamilton cycles in \cite{O1960} states that every graph $G$ with $\sigma_2(G)\geq |G|$ has a Hamilton cycle, and Brandt et al.~\cite{BCFGL1997} extended this result as follows.

\begin{theorem}[\cite{BCFGL1997}]\label{thm:sigma2 kcycle}
    Let $G$ be a graph of order $n$ and let $k$ be a positive integer.
    If $n\geq 4k$ and $\sigma_2(G)\geq n$, then $G$ has a 2-factor with exactly $k$ cycles.
\end{theorem}

As an extension of Ore's theorem to the minimum degree sum condition, Bondy~\cite{B1980} showed that if a $\kappa$-connected graph $G$ satisfies $\sigma_{\kappa+1}(G)>\frac{1}{2}(\kappa+1)(|G|-1)$, then $G$ has a Hamilton cycle.
This result was also extended to a sufficient condition for a graph to have a 2-factor with a prescribed number of cycles in Chiba~\cite{C2018}.

\begin{theorem}[\cite{C2018}]\label{thm:sigmakappa kcycle}
    Let $G$ be a $\kappa$-connected graph of order $n$ and let $k$ be a positive integer.
    If $n\geq 5k-2$ and $\sigma_{\left\lceil \kappa/k\right\rceil+1}(G)\geq \frac{1}{2}(\left\lceil \kappa/k \right\rceil+1)(n-1)$, then $G$ has a 2-factor with exactly $k$ cycles.
\end{theorem}

The Chv\'{a}tal-Erd\H{o}s condition for a graph to have a 2-factor with prescribed number of cycles has been investigated.
Chv\'{a}tal and Erd\H{o}s~\cite{CE1972} showed that every graph $G$ with $\alpha(G)\leq \kappa(G)$ has a Hamilton cycle, where $\kappa(G)$ is the connectivity of $G$.
The assumption $\alpha(G)\leq \kappa(G)$ is called \emph{the Chv\'{a}tal-Erd\H{o}s condition}.
Kaneko and Yoshimoto~\cite{KY2003} showed that every 4-connected graph of order at least six with the Chv\'{a}tal-Erd\H{o}s condition has a 2-factor with two cycles, and conjectured the same conclusion for 2-connected graphs of sufficiently large order with the Chv\'{a}tal-Erd\H{o}s condition.
Chen et al.~\cite{CGKOSS2007} showed that for any given integer $k\geq 1$, every 2-connected graph with sufficiently large order has a 2-factor with exactly $k$ cycle, where the assumption on the order depends on $k$ and the independence number.

We refer the reader to Chiba and Yamashita~\cite{CY2018} for a comprehensive survey on sufficient conditions for a graph to have a 2-factor with a prescribed number of cycles.
In this paper, we consider new type assumptions for a graph to have a 2-factor, which relates to both degree conditions and the Chv\'{a}tal-Erd\H{o}s condition.

\subsection{New type degree conditions for a graph to have a 2-factor}

In a graph $G$ with the Chv\'{a}tal-Erd\H{o}s condition, the order of all independent sets are bounded by $\kappa(G)$, which is at most the minimum degree of $G$.
On the other hand, in proofs of some results, it is usual to consider one specified independent set and use the assumption to the independent set.
Therefore, we introduce the assumption that the order of each independent set is bounded by a function of the minimum degree of the independent set.

For a graph $G$, let $\mathcal{I}(G)$ be the family of independent sets of $G$.
For an independent set $I\in \mathcal{I}(G)$, let $\delta_G(I)$ be the minimum degree of the vertices contained in $I$.
The following are the first two results of this paper.

\begin{theorem}\label{thm:sigma0}
    Let $G$ be a graph.
    If every independent set $I$ of $G$ satisfies $|I|\leq \delta_G(I)-1$, then $G$ has a 2-factor.
\end{theorem}

\begin{theorem}\label{thm:sigma1}
    Let $G$ be a graph.
    If every independent set $I$ of $G$ satisfies $|I|\leq \delta_G(I)$, then either $G$ has a 2-factor or $G$ is isomorphic to a graph in $\mathcal{G}$, where $\mathcal{G}$ is defined below.
\end{theorem}

We define the set $\mathcal{G}$ of graphs as follows.
For a positive integer $\ell$, let $\mathcal{G}_{\ell}$ be the set of all graphs obtained by the join of a graph of order $\ell-1$ and the union of $\ell$ disjoint copies of the complete graph $K_2$.
The set $\mathcal{G}$ is defined as $\bigcup_{\ell=1}^{\infty}\mathcal{G}_\ell$.
Note that every graph in $\mathcal{G}$ does not have a 2-factor.

The following propositions show the relation between the assumptions of Theorems~\ref{thm:sigma0} and \ref{thm:sigma1} and known results.

\begin{proposition}\label{prop:sigma0 degree relation}
    If a graph $G$ satisfies $\delta(G)\geq \frac{|G|}{2}$, then either $|I|\leq \delta_G(I)-1$ for every independent set $I$ of $G$, or $|G|$ is even and $G$ contains a balanced complete bipartite graph as a spanning subgraph.
\end{proposition}

\begin{proof}
    Suppose that $G$ satisfies $\delta(G)\geq \frac{|G|}{2}$ and that there is an independent set $I$ of $G$ with $|I|\geq \delta_G(I)$.
    For each vertex $v\in I$, we have $\frac{|G|}{2}\leq d_G(v)\leq |V(G)\setminus I|=|G|-|I|$, which implies that $|I|\leq \frac{|G|}{2}$.
    On the other hand, we have $|I|\geq \delta_G(I)\geq \delta(G)\geq \frac{|G|}{2}$.
    These two force $|I|=|V(G)\setminus I|=\frac{|G|}{2}$ and each vertex of $I$ is adjacent to all vertices of $G-I$.
    Thus, $G$ contains a balanced complete bipartite graph with the partite sets $I$ and $V(G)\setminus I$.
\end{proof}

Note that the case where $G$ contains the complete bipartite graph $K_{|G|/2,|G|/2}$ is easy to deal with in terms of the existence of a 2-factor.
Thus, the assumption that $|I|\leq \delta_G(I)-1$ for every independent set $I$ of $G$ is a relaxation of the minimum degree condition $\delta(G)\geq \frac{|G|}{2}$.

\begin{proposition}\label{prop:sigma1 ce relation}
    If a graph $G$ satisfies $\alpha(G)\leq \kappa(G)$, then $|I|\leq \delta_G(I)$ for every independent set $I$ of $G$.
\end{proposition}

\begin{proof}
    Suppose that $G$ satisfies $\alpha(G)\leq \kappa(G)$.
    Let $I$ be an arbitrary independent set of $G$.
    Since $\delta(G)\geq \kappa(G)$, we have $|I|\leq \alpha(G)\leq \kappa(G)\leq \delta(G)\leq \delta_G(I)$.
\end{proof}

Thus, the assumption that $|I|\leq \delta_G(I)$ for every independent set $I$ of $G$ is a relaxation of the Chv\'{a}tal-Erd\H{o}s condition.

\subsection{2-factors with bounded number of cycles}

By Theorem~\ref{thm:sigma0}, if every independent set $I$ of a graph $G$ satisfies $|I|\leq \delta_G(I)-1$, then $G$ has a 2-factor.
However, the number of cycles in a 2-factor cannot be bounded by the assumption since, for any positive integer $\ell$, the union of $\ell$ disjoint copies of the complete graph $K_{\ell+2}$ satisfies the assumption of Theorem~\ref{thm:sigma0}.
Then, what is a reasonable additional condition for $G$ with the assumption of Theorem~\ref{thm:sigma0} to have a 2-factor with at most $k$ cycles?
By Ore's theorem in \cite{O1960}, if $G$ satisfies $\sigma_2(G)\geq |G|$, then $G$ has a 2-factor with one cycle. 
Motivated by the result,  we pose the following conjecture.

\begin{conjecture}\label{conj:sigma0 k comp}
    Let $G$ be a graph of order $n$ and let $k$ be a positive integer.
    If $\sigma_{k+1}(G)\geq n$ and every independent set $I$ of $G$ satisfies $|I|\leq \delta_G(I)-1$, then $G$ has a 2-factor with at most $k$ cycles.
\end{conjecture}

If Conjecture~\ref{conj:sigma0 k comp} is true, then the lower bound of $\sigma_{k+1}(G)$ is best possible.
Indeed, for a positive integer $k$, let $H_0$ be a complete graph of order $k+3$ with vertices $\{v_1,v_2,\dots , v_{k+3}\}$ and let $H_1, H_2, \dots , H_k$ be $k$ disjoint copies of the complete graph $K_{k+2}$.
The graph $G_k$ is obtained from $H_0, H_1, \dots , H_k$ by joining all the vertices of $H_i$ and the vertex $v_i\in V(H_0)$ for each $i\in [k]$.
Then we have $|G_k|=k+3+k(k+2)=k^2+3k+3$.
Since $\delta(G_k)=k+2$ and $\alpha(G_k)=k+1$, every independent set $I$ of $G_k$ satisfies $|I|\leq \alpha(G_k)=\delta(G_k)-1\leq \delta_{G_k}(I)-1$.
Let $u_i$ be a vertex of $H_i$ for each $i\in [k]$.
Considering an independent set $I=\{v_{k+1},u_1,u_2,\dots ,u_k\}$ of $G_k$, we infer that $\sigma_{k+1}(G_k)=(k+1)\delta(G_k)=(k+1)(k+2)=n-1$.
On the other hand, since each of $\{v_1, v_2, \dots , v_k\}$ is a cut vertex of $G_k$ and $G-\{v_1,v_2,\dots , v_k\}$ has $k+1$ components, every 2-factor of $G_k$ consists of at least $k+1$ cycles.

If a graph $G$ contains a cut set $S\subseteq V(G)$ such that $\omega(G-S)\geq |S|+k-1$ where $\omega(G-S)$ is the number of components of $G-S$, then obviously every 2-factor of $G$ consists of at least $k$ cycles.
On the other hand, we can show that Conjecture~\ref{conj:sigma0 k comp} holds for graphs with such cut sets.

\begin{theorem}\label{thm:sigma0 k comp}
    Let $G$ be a graph of order $n$ and let $k$ be a positive integer.
    Suppose that $\sigma_{k+1}(G)\geq n$ and that every independent set $I$ of $G$ satisfies $|I|\leq \delta_G(I)-1$.
    If there is a cut set $S$ of $G$ with $\omega(G-S)\geq |S|+k-1$, then $G$ has a 2-factor with at most $k$ cycles.
\end{theorem}

When we set $k=2$, the assumption that there is a cut set $S$ of $G$ with $\omega(G-S)\geq |S|+k-1=|S|+1$ is equivalent to the assumption that $G$ is not 1-tough.
Bauer et al.~\cite{BVMS1989} showed the following result on a degree sum condition for a 1-tough graph to have a dominating cycle.
A cycle $C$ of a graph $G$ is called \emph{dominating} if $G-V(C)$ has no edge.

\begin{theorem}[\cite{BVMS1989}]\label{thm:1 tough}
    Let $G$ be a graph of order $n$.
    If $G$ is 1-tough and $\sigma_3(G)\geq n$, then every longest cycle of $G$ is a dominating cycle.
\end{theorem}

Combining Theorems~\ref{thm:sigma0 k comp} and \ref{thm:1 tough}, we have the following corollary, which is a partial result of Conjecture~\ref{conj:sigma0 k comp}.

\begin{corollary}\label{cor:sigma0 2 comp}
    Let $G$ be a graph of order $n$.
    If $\sigma_3(G)\geq n$ and every independent set $I$ of $G$ satisfies $|I|\leq \delta_G(I)-1$, then $G$ has a 2-factor with at most two cycles.
\end{corollary}

The paper is organized as follows:
We give proofs of Theorems~\ref{thm:sigma0} and \ref{thm:sigma1} in Section~\ref{sec:existence}, and proofs of Theorem~\ref{thm:sigma0 k comp} and Corollary~\ref{cor:sigma0 2 comp} in Section~\ref{sec:k components}.

Before we go to the proofs, we introduce the notation used in our proof.
For a cycle $C$ of a graph $G$, we set one cyclic orientation to $C$, and let $\overrightarrow{C}$ denote the cycle $C$ including its orientation.
For a cycle $\overrightarrow{C}$ of $G$ and a vertex $v\in V(C)$, let $v^+$ (resp. $v^-$) denote the vertex that comes after (resp. before) $v$ along $\overrightarrow{C}$.
For a cycle $\overrightarrow{C}$ of $G$ and two distinct vertices $u$ and $v$ of $C$, let $u\overrightarrow{C}v$ (resp. $u\overleftarrow{C}v$) denote the subpath of $C$ from $u$ to $v$ along the orientation (resp. opposite orientation) of $\overrightarrow{C}$.

\section{Existence of a 2-factor}\label{sec:existence}

In this section, we give proofs of Theorems~\ref{thm:sigma0} and \ref{thm:sigma1}.

\subsection{Proof of Theorem~\ref{thm:sigma0}}

The following lemma is the essential part of our proof of Theorem~\ref{thm:sigma0}.

\begin{lemma}\label{lem:1 vertex inclusion}
    Let $G$ be a graph such that every independent set $I$ of $G$ satisfies $|I|\leq \delta_G(I)-1$.
    Suppose that $S$ is a set of vertices such that $G-S$ has a 2-factor with $k$ cycles.
    If there is a vertex $v\in S$ such that $|N_G(v)\cap S|\leq 1$, then $G-S\setminus \{v\}$ has a 2-factor with at most $k$ cycles.
\end{lemma}

\begin{proof}
    Suppose that $G-S$ has a 2-factor with $k$ disjoint cycles $C_1, C_2, \dots , C_k$.
    We fix one orientation for each cycle $C_i$, and let $\overrightarrow{C_i}$ denote the cycle including its orientation.
    Let $X=N_G(v)\setminus S$.
    For each $x\in X$, let $x^+$ be the succeeding vertex of $x$ along the cycle in $\{C_1,C_2,\dots , C_k\}$ that contains $x$, and let $X^+=\{x^+\mid x\in X\}$.
    If $X^+\cup \{v\}$ is an independent set of $G$, then $|X^+\cup \{v\}|=|X^+|+1=|N_G(v)\setminus S|+1\geq d_G(v)-1+1\geq \delta_G(X^+\cup \{v\})$, a contradiction by the assumption of the lemma.
    Thus, $X^+\cup \{v\}$ is not an independent set, and hence either $vx^+\in E(G)$ for some $x^+\in X^+$ or $x_1^+x_2^+\in E(G)$ for some $x_1^+,x_2^+\in X^+$.
    
    Suppose that $vx^+\in E(G)$ for some $x^+\in X^+$.
    Without loss of generality, we may assume that $x^+\in V(C_1)$.
    By setting $C_1'=x^+\overrightarrow{C_1}xvx^+$, $C_1'\cup \bigcup_{i=2}^k C_k$ is a 2-factor of $G-S\setminus\{v\}$ with $k$ cycles, as desired.

    Next, we assume that $x_1^+x_2^+\in E(G)$ for some $x_1^+,x_2^+\in X^+$.
    If $x_1^+$ and $x_2^+$ lie in the same cycle, say $C_1$, then we set $C_1'=x_1^+\overrightarrow{C_1}x_2vx_1\overrightarrow{C_1}x_2^+x_1^+$ to conclude that $C_1'\cup \bigcup_{i=2}^k C_k$ is a 2-factor of $G-S\setminus\{v\}$ with $k$ cycles, as desired.
    Otherwise, without loss of generality, we may assume that $x_1^+\in V(C_1)$ and $x_2^+\in V(C_2)$.
    Let $C'=x_1^+\overrightarrow{C_1}x_1vx_2\overleftarrow{C_2}x_2^+x_1^+$.
    Then $C'\cup \bigcup_{i=3}^k C_k$ is a 2-factor of $G-S\setminus\{v\}$ with $k-1$ cycles, as desired.
    This completes the proof of Lemma~\ref{lem:1 vertex inclusion}.
\end{proof}

Now we prove Theorem~\ref{thm:sigma0} by using Lemma~\ref{lem:1 vertex inclusion}.

\begin{proof}[Proof of Theorem~\ref{thm:sigma0}]
    Let $G$ be a graph such that every independent set $I$ of $G$ satisfies $|I|\leq \delta_G(I)-1$.
    Suppose that $G$ does not have a 2-factor.
    If there is a vertex $v\in V(G)$ such that $d_G(v)\leq 1$, then the independent set $I=\{v\}$ satisfies $|I|=1\geq d_G(v)=\delta_G(I)$, a contradiction.
    Hence, we infer that $\delta(G)\geq 2$, and in particular, $G$ contains a cycle.
    Let $G'$ be a maximum subgraph of $G$ having a 2-factor, and let $S=V(G)\setminus V(G')$.
    By the maximality of $G'$, $G[S]$ has no cycle, and hence there is a vertex $v\in S$ such that $|N_G(v)\cap S|\leq 1$.
    Thus, by Lemma~\ref{lem:1 vertex inclusion}, $G-S\setminus\{v\}= G[V(G')\cup \{v\}]$ has a 2-factor, which contradicts the maximality of $G'$.
\end{proof}

\subsection{Tutte's 2-factor theorem}

In our proof of Theorem~\ref{thm:sigma1}, we use a well-known result by Tutte~\cite{T1952} which gives a criterion for a graph to have a 2-factor.

For a graph $G$ and disjoint sets of vertices $S\subseteq V(G)$ and $T\subseteq V(G)$, let $e(S,T)$ denote the number of edges joining a vertex of $S$ and a vertex of $T$.
In particular, we abbreviate $e(\{v\},S)$ to $e(v,S)$ for a vertex $v\in V(G)\setminus S$.
For a graph $G$ and disjoint sets of vertices $S\subseteq V(G)$ and $T\subseteq V(G)$, an \emph{odd component} (resp. \emph{even component}) of $(G;S,T)$ is a component $D$ of $G-S\cup T$ such that $e(T,V(D))$ is odd (resp. even).
Let $\mathcal{H}_G(S,T)$ denote the family of odd components of $(G;S,T)$, and let $h_G(S,T)=|\mathcal{H}_G(S,T)|$.
Tutte~\cite{T1952} proved the following theorem.

\begin{theorem}[\cite{T1952}]\label{thm:tutte}
    A graph $G$ has a 2-factor if and only if 
    \[\delta(S,T):=2|S|-2|T|+\sum_{v\in T}d_{G-S}(v)-h_G(S,T)\geq 0\]
    for every pair $(S,T)$ of disjoint sets of vertices of $G$.
\end{theorem}

A pair of disjoint sets of vertices $(S,T)$ with $\delta(S,T)<0$ is called a \emph{barrier} of $G$.
As a result of Theorem~\ref{thm:tutte}, if a graph $G$ does not have a 2-factor, then there is a barrier of $G$.
It is known that $\delta(S,T)$ is even for any pair $(S,T)$, and thus $\delta(S,T)\leq -2$ for any barrier of $G$.
A barrier $(S,T)$ of $G$ is called \emph{minimal} if $(S',T')$ is not a barrier of $G$ for any pair of disjoint sets of vertices with $|S'\cup T'|<|S\cup T|$.
The following properties of a minimal barrier often appear in literature. (For example, see Lemma 1 in Aldred et al.~\cite{AEFOS2011}.) 

\begin{lemma}[\cite{AEFOS2011}]\label{lem:minimal barrier}
    Let $G$ be a graph without 2-factor and let $(S,T)$ be a minimal barrier of $G$.
    Then the followings hold.
    \begin{enumerate}
        \item $T$ is an independent set of $G$,
        \item for every even component $D$ of $(G;S,T)$, $e(T,V(D))=0$,
        \item for every odd component $D$ of $(G;S,T)$ and every $v\in T$, $e(v,V(D))\leq 1$, and
        \item $|T|>|S|$.
    \end{enumerate}
\end{lemma}

\subsection{Proof of Theorem~\ref{thm:sigma1}}

Suppose that $|I|\leq \delta_G(I)$ for every independent set $I$ of $G$ and that $G$ does not have a 2-factor.
We shall show that $G\in\mathcal{G}$.

Let $\mathcal{I}(G)$ be the family of independent sets of $G$.
By Theorem~\ref{thm:tutte}, $G$ contains a barrier.
Let $(S,T)$ be a minimal barrier of $G$.
Let $t=|T|$ and $T=\{v_1, v_2, \dots , v_t\}$.

Since $T\in \mathcal{I}(G)$ by Lemma~\ref{lem:minimal barrier}, we have $\delta_G(T)\geq |T|=t$, and thus $d_G(v)\geq t$ for every vertex $v\in T$.
Without loss of generality, we may assume that $d_{G-S}(v_1)\leq d_{G-S}(v_i)$ for all $i\in \{2,3,\dots , t\}$.
Let $a:=t-|S|>0$.

Suppose first that $d_{G-S}(v_1)\geq a+1$.
Then we have $\sum_{v\in T}d_{G-S}(v)\geq (a+1)t$ and hence
\begin{align*}
  -2\geq \delta(S,T)=2|S|-2|T|+\sum_{v\in T}d_{G-S}(v)-h_G(S,T)\geq -2a+(a+1)t-h_G(S,T),
\end{align*}
which implies that
\begin{align}
  h_G(S,T)\geq (a+1)t-2(a-1)=(a-1)(t-2)+2t\geq 2t>|S|+t,\label{eq:2factor tutte1}
\end{align}
where the third inequality holds since $1\leq a\leq t$.

If every odd component $D\in\mathcal{H}_G(S,T)$ has at least $t+1$ vertices, then Lemma~\ref{lem:minimal barrier} implies that there is a vertex $x_D\in V(D)\setminus N_G(T)$ for every $D\in \mathcal{H}_G(S,T)$, 
and hence $I=T\cup \{x_D\mid D\in\mathcal{H}_G(S,T)\}$ is an independent set of $G$ with $|I|=t+h_G(S,T)>|S|+h_G(S,T)\geq d_G(v_1)\geq \delta_G(I)$, a contradiction.
Thus, there is an odd component $D_1\in \mathcal{H}_G(S,T)$ such that $|V(D_1)|\leq t$.
Let $x_1$ be a vertex of $D_1$ such that $e(x_1,T)$ is minimum among all vertices of $D_1$.
Then we have 
\begin{align*}
  d_{G-S}(x_1)=d_{D_1}(x)+e(x_1,T)\leq |V(D_1)|-1+\frac{t}{|V(D_1)|}\leq t,
\end{align*}
where the last inequality follows from $1\leq |V(D_1)|\leq t$.
Thus, $d_G(x_1)\leq |S|+d_{G-S}(x_1)\leq |S|+t$.
Now we let $I'=\{x_1\}\cup \{x_D\mid D\in \mathcal{H}_G(S,T)\setminus\{D_1\}\}$ where $x_D\in V(D)$ is an arbitrary vertex for each $D\in \mathcal{H}_G(S,T)\setminus\{D_1\}$.
It is obvious that $I'$ is an independent set of $G$, and by (\ref{eq:2factor tutte1}), we have $|I'|=h_G(S,T)> |S|+t\geq d_G(x_1)\geq \delta_G(I')$, a contradiction.
Hence, we conclude that $d_{G-S}(v_1)\leq a$.

Since $d_{G-S}(v_1)\geq d_G(v_1)-|S|\geq t-|S|=a$, we have $d_{G-S}(v_1)=a$ and thus $d_G(v_1)=t$.
If there is a vertex $x\in V(G)\setminus(S\cup T\cup N_G(T))$, then $T\cup \{x\}$ is an independent set of $G$ with $|T\cup\{x\}|>t=d_G(v_1)\geq \delta_G(T\cup \{x\})$, a contradiction.
Thus, $V(G)=S\cup T\cup N_G(T)$.
In particular, there is no even component of $(G;S,T)$ by Lemma~\ref{lem:minimal barrier}.
By Lemma~\ref{lem:minimal barrier}, there are odd components $D_1, D_2, \dots , D_a$ of $(G;S,T)$ such that $|N_G(v_1)\cap V(D_i)|=1$ for every $i\in [a]$.
We set $N_G(v_1)\cap V(D_i)=\{x_i\}$.
Let $\mathcal{H}_G'(S,T)=\mathcal{H}_G(S,T)\setminus\{D_1, D_2, \dots , D_a\}$.
If there are more than $t-1$ odd components in $\mathcal{H}_G'(S,T)$, then $\{v_1\}\cup\{x_D\mid D\in \mathcal{H}_G'(S,T)\}$ is an independent set of $G$ where $x_D$ is an arbitrary vertex of $D$, and this violates the assumption.
Thus, we have $h_G(S,T)=a+|\mathcal{H}_G'(S,T)|\leq a+t-1$.
Since $(S,T)$ is a barrier of $G$, we have
\begin{align*}
  -2\geq \delta(S,T)=2|S|-2|T|+\sum_{v\in T}d_{G-S}(v)-h_G(S,T)
  &\geq -2a+\sum_{v\in T}d_{G-S}(v)-(a+t-1)\\
  &=-3a-t+1+\sum_{v\in T}d_{G-S}(v),
\end{align*}
which implies that $\sum_{v\in T}d_{G-S}(v)\leq t+3a-3$.
By the choice of $v_1$, we have $at\leq \sum_{v\in T}d_{G-S}(v)\leq t+3a-3$, and thus
\begin{align}
  (t-3)(a-1)=at-t-3a+a\leq 0.\label{eq:2factor tutte2}
\end{align}
This together with the fact $1\leq a\leq t$ implies that either $a=1$, or $(t,a)\in\{(2,2), (3,2)\}$.
Furthermore, if $a=1$ or $t=3$, then (\ref{eq:2factor tutte2}) holds by equalities.

Suppose first that $(t,a)=(2,2)$.
Then we have $S=\emptyset$ and $d_G(v_1)=a=2$.
Let $D_1$ and $D_2$ be odd components of $(G;S,T)$ such that there is a vertex $x_i\in N_G(v_1)\cap V(D_i)$ for $i\in \{1,2\}$.
If both $D_1$ and $D_2$ have at least 2 vertices, then there is a vertex $y_i\in V(D_i)\setminus\{x_i\}$ for each $i\in \{1,2\}$ and $\{v_1, y_1, y_2\}$ is an independent set of $G$ that violates the assumption.
Hence, without loss of generality, we may assume that $V(D_1)=\{x_1\}$.
Since $G$ has the minimum degree at least 2, $x_1$ has a neighbor other than $v_1$, which must be $v_2$.
This implies that $e(T, V(D_1))=2$, a contradiction by the assumption that $D_1$ is an odd component of $(G;S,T)$.

Next, we assume that $(t,a)=(3,2)$.
Since (\ref{eq:2factor tutte2}) holds by equality, we know that $h_G(S,T)=2+3-1=4$ and $d_{G-S}(v)=2$ for every $v\in T$.
Let $\mathcal{H}_G(S,T)=\{D_1, D_2, D_3, D_4\}$.
We have $e(T, V(D_1)\cup V(D_2)\cup V(D_3)\cup V(D_4))=6$ and one of the odd components, say $D_1$, satisfies $e(T, V(D_1))=1$.
Let $V(D_1)=\{x_1\}$, and without loss of generality, we may assume that $v_1x_1\in E(G)$.
Since $|S|=1$, we have $d_G(x_1)=2$, and thus $\{x_1, v_2, v_3\}$ is an independent set that violates the assumption, a contradiction.

Thus, we have $a=1$.
Since (\ref{eq:2factor tutte2}) holds by equality, it follows that $h_G(S,T)=a+t-1=t$ and $d_{G-S}(v)=1$ for every $v\in T$.
Hence, we may assume that $\mathcal{H}_G(S,T)=\{D_1, D_2, \dots , D_t\}$ and $N_G(v_i)\setminus S=\{x_i\}\subseteq V(D_i)$ for every $i\in [t]$.
As $\bigcup_{i=1}^{t}V(D_i)\subseteq N_G(T)$, we know that $V(D_i)=\{x_i\}$ for every $i\in [t]$.
Since $\{x_1, x_2, \dots , x_t\}$ is an independent set of $G$, we have $d_G(x_i)\geq t$ for each $i\in [t]$.
Thus, each vertex of $T\cup\{x_1, x_2, \dots , x_t\}$ is adjacent to all vertices of $S$, implying that $G$ belongs to the class $\mathcal{G}_t$.
This completes the proof of Theorem~\ref{thm:sigma1}.

\section{2-factor with a bounded number of components}\label{sec:k components}

\subsection{Proof of Theorem~\ref{thm:sigma0 k comp}}

\setcounter{claim}{0}

The case $k=1$ follows from Ore's theorem in \cite{O1960}.
We fix an integer $k\geq 2$.
Let $G$ be a graph of order $n$ which satisfies the assumptions of Theorem~\ref{thm:sigma0 k comp}.
We take a cut set $S=\{v_1, v_2, \dots , v_s\}$ of $G$ such that $\omega(G-S)\geq s+k-1$.
Let $\{G_1,G_2,\dots ,G_{s+k-1+j}\}$ be the set of components of $G-S$ where $j$ is a non-negative integer, and let $n_i=|G_i|$ for each $i\in [s+k-1+j]$.

\begin{claim}\label{cl:2factor comp mindegree}
  $\delta(G-S)\geq k$. In particular, for every $i\in [s+k-1+j]$, $n_i\geq k+1$.
\end{claim}

\begin{proof}
  Since $\omega(G-S)\geq s+k-1$, for every vertex $x\in V(G)\setminus S$, there is an independent set $I_x$ of $G$ with $s+k-1$ vertices, including $x$.
  We have $d_G(x)\geq \delta_G(I_x)\geq |I_x|+1=s+k$, which implies that $d_{G-S}(x)\geq d_G(x)-s\geq k$.
  Thus, $\delta(G-S)\geq k$ and every component of $G-S$ has at least $k+1$ vertices.
\end{proof}

Now we consider the case $s=1$.
Suppose first that $j\geq 1$, and thus $s+k-1+j\geq k+1$. 
For each $i\in [k+1]$, we arbitrarily choose $x_i\in V(G_i)$ and consider an independent set $\{x_1, x_2, \dots , x_{k+1}\}$ of $G$.
Since $d_G(x_i)\leq n_i-1+s=n_i$, we have $\sigma_{k+1}(G)\leq \sum_{i=1}^{k+1} d_G(x_i)\leq \sum_{i=1}^{k+1} n_i\leq n-1$, a contradiction.
Hence, we have $j=0$ and thus $G-S$ has exactly $k$ components.
If every component of $G-S$ is Hamiltonian, then obviously $G-v_1$ contains a 2-factor with exactly $k$ cycles, and $G$ has a 2-factor with at most $k$ cycles by Lemma~\ref{lem:1 vertex inclusion}. 
Thus, we may assume that there is a non-Hamiltonian component of $G-v_1$.
Without loss of generality, we may assume that $G_1$ is not Hamiltonian.
As $n_1\geq k+1\geq 3$, $G_1$ is not complete, and thus there are non-adjacent vertices $x_1$ and $y_1$ of $G_1$ such that $d_{G_1}(x_1)+d_{G_1}(y_1)=\sigma_2(G_1)\leq n_1-1$ by a result in Ore~\cite{O1960}.
For each $i\in \{2,3,\dots ,k\}$, let $x_i$ be an arbitrary vertex of $G_i$.
By our assumption of $\sigma_{k+1}(G)\geq n$, we have
\begin{align*}
  \sum_{i=1}^{k} n_i+1=n&\leq d_G(x_1)+d_G(y_1)+\sum_{i=2}^{k}d_G(x_i)\\
  &\leq (d_{G_1}(x)+1)+(d_{G_1}(y)+1)+\sum_{i=2}^{k}n_i \leq \sum_{i=1}^{k}n_i+1.
\end{align*}
Thus, all inequalities above hold by equalities, which derives the followings;
\begin{itemize}
  \item $\sigma_2(G_1)=n_1-1$,
  \item if non-adjacent vertices $x$ and $y$ of $G_1$ satisfies $d_{G_1}(x)+d_{G_1}(y)=n_1-1$, then $x,y\in N_G(v_1)$, and
  \item for every $i\in \{2,3,\dots k\}$, $G_i$ is a complete graph of order at least $k+1\geq 3$.
\end{itemize}
The first two conditions imply that $G[V(G_1)\cup\{v_1\}]$ is Hamiltonian, and the last condition implies that $G_i$ is Hamiltonian for every $i\in \{2,3,\dots k\}$.
Combining these, we obtain a 2-factor of $G$ with exactly $k$ cycles.

In the rest of the proof, we may assume that $s\geq 2$, and thus $s+k-1+j\geq k+1$.

\begin{claim}\label{cl:2factor comp jequal0}
  $j=0$ and $s\leq k+1$.
  Furthermore, if $s\geq 3$, then every component of $G-S$ is Hamilton-connected.
\end{claim}

\begin{proof}
  For each $i\in [k+1]$, let $x_i$ be a vertex of $G_i$.
  Since $\sigma_{k+1}(G)\geq n$, we have
  \begin{align*}
    \sum_{i=1}^{s+k-1+j} n_i+s=n\leq \sum_{i=1}^{k+1} d_G(x_i)\leq \sum_{i=1}^{k+1} (n_i-1+s)=\sum_{i=1}^{k+1} n_i+(k+1)(s-1),
  \end{align*}
  so
  \begin{align}
    (k+1)(s-1)-s\geq \sum_{i=k+2}^{s+k-1+j}n_i\geq (k+1)(s+j-2)\label{eq:2factor comp cl2}
  \end{align}
  where the last inequality follows from Claim~\ref{cl:2factor comp mindegree}.
  Thus, we have $(k+1)j\leq k+1-s$, which implies $j=0$ and $s\leq k+1$.

  When $s\geq 3$, we have $s+k-1\geq k+2$.
  Without loss of generality, we may assume that $n_i\leq n_{s+k-1}$ for every $1\leq i\leq s+k-1$.
  By the first inequality of \ref{eq:2factor comp cl2}, we have 
  \begin{align*}
    \sum_{i=k+2}^{s+k-1}(n_i-k-1)=\sum_{i=k+2}^{s+k-1}n_i-(k+1)(s-2)\leq (k+1)(s-1)-s-(k+1)(s-2)=k-s+1.
  \end{align*}
  Since $n_i-k-1\geq 0$ by Claim~\ref{cl:2factor comp mindegree} for every $i\in [s+k-1]$, $n_{s+k-1}\leq k+1+\sum_{i=k+2}^{s+k-1}(n_i-k-1)\leq k+1+k-s+1\leq 2k-1$.
  Thus, $n_i\leq n_{s+k-1}\leq 2k-1$ for every $1\leq i\leq s+k-1$.
  Again by Claim~\ref{cl:2factor comp mindegree}, for each $i\in [s+k-1]$, it follows that $\sigma_2(G_i)\geq 2\delta(G-S)\geq 2k\geq n_i+1$, which implies that $G_i$ is Hamilton-connected by Ore's theorem in \cite{O1960}.
\end{proof}

We first consider the case where there is a component of $G-S$ that is not Hamilton-connected.
By Claim~\ref{cl:2factor comp jequal0}, we know $s=2$, and thus $G-S$ has exactly $k+1$ components.
Without loss of generality, we may assume that $G_1$ is not Hamilton-connected.
As $n_1\geq k+1\geq 3$, $G_1$ is not complete and thus $\sigma_2(G_1)\leq n_1$.
Let $x_1$ and $y_1$ be non-adjacent vertices of $G_1$ such that $d_{G_1}(x_1)+d_{G_1}(y_1)=\sigma_2(G_1)\leq n_1$.
We fix $i\in \{2,3,\dots , k+1\}$ arbitrarily, and for each $j\in [k+1]\setminus\{1,i\}$, let $x_j$ be an arbitrary vertex of $G_j$.
By the assumption that $\sigma_{k+1}(G)\geq n$, we have
\begin{align*}
  \sum_{j=1}^{k+1} n_j+2=n&\leq d_G(x_1)+d_G(y_1)+\sum_{j\in [k+1]\setminus\{1,i\}} d_G(x_j)\\
  &\leq (d_{G_1}(x_1)+2)+(d_{G_1}(y_1)+2)+\sum_{j\in[k+1]\setminus\{1,i\}} (n_j-1+2)\\
  &\leq \sum_{j\in [k+1]\setminus\{i\}} n_j+k+3,
\end{align*}
which forces $n_i\leq k+1$.
Since $n_i\geq k+1$ by Claim~\ref{cl:2factor comp mindegree}, every inequality above holds by equality, which implies that $\sigma_2(G_1)=n_1$ and hence $G_1$ is Hamiltonian.
Furthermore, since the choice of $i\in [k+1]$ is arbitrary, we conclude that $G_i$ is complete and that $V(G_i)\subseteq N_G(v_1)\cap N_G(v_2)$ for every $i\in \{2,3,\dots ,k\}$.
Combining a Hamilton cycle of $G[\{v_1, v_2\}\cup V(G_2)\cup V(G_3)]$ and Hamilton cycles of other components $G_1, G_4, G_5, \dots , G_{k+1}$, we obtain a 2-factor of $G$ with exactly $k$ components.

In the rest of the proof, we assume that every component of $G-S$ is Hamilton-connected.
We define a bipartite graph $H$ by $V(H)=S\cup\{u_1, u_2, \dots , u_{s+k-1}\}$ and $E(H)=\{v_iu_j\mid N_G(v_i)\cap V(G_j)\neq\emptyset\}$.

\begin{claim}\label{cl:2factor comp Hcycle}
  $H$ has a cycle that contains all vertices of $S$.
\end{claim}

\begin{proof}
  For each $i\in [k+1]$, let $x_i$ be a vertex of $G_i$, and we consider an independent set $\{x_1, x_2, \dots , x_{k+1}\}$ of $G$.
  Since $d_G(x_i)\leq n_i-1+d_H(u_i)$ for each $i\in [k+1]$, by the assumption $\sigma_{k+1}(G)\geq n$, we have
  \begin{align*}
    \sum_{i=1}^{s+k-1} n_i+s=n\leq \sum_{i=1}^{k+1} d_G(x_i)\leq \sum_{i=1}^{k+1} (n_i-1+d_H(u_i))=\sum_{i=1}^{k+1} n_i-k-1+\sum_{i=1}^{k+1} d_H(u_i),
  \end{align*}
  and hence
  \begin{align*}
    \sum_{i=1}^{k+1} d_H(u_i)\geq \sum_{i=k+2}^{s+k-1}n_i+k+1+s\geq (k+1)(s-2)+k+1+s=(k+1)(s-1)+s.
  \end{align*}
  Since $d_H(u_i)\leq s$ for every $i\in [k+1]$, we infer that at least $s$ vertices of $\{u_1,u_2,\dots , u_{k+1}\}$ have degree exactly $s$ in $H$.
  Combining $S$ and such vertices, we obtain a subgraph of $H$ that is isomorphic to the complete bipartite graph $K_{s,s}$.
  It is obvious that the subgraph contains a desired cycle of $H$.
\end{proof}

Let $C_0$ be a cycle of $H$ as in Claim~\ref{cl:2factor comp Hcycle}.
Without loss of generality, we may assume that the vertices of $S$ appear in the order $v_1, v_2, \dots , v_s$ along $C_0$.

\begin{claim}\label{cl:2factor comp cycle reconstruct}
  Let $C=v_1u_{a_1}v_2u_{a_2}v_3\cdots v_su_{a_s}v_1$ be a cycle of $H$ where $a_1, a_2, \dots , a_s$ are pairwise distinct $s$ elements of $[s+k-1]$.
  For each $i\in [s]$, there is an index $b_i\in [s+k-1]\setminus\{a_j\mid 1\leq j\leq s, j\neq i\}$ and a $v_iv_{i+1}$-path $Q_i$ of $G$ such that $V(Q_i)=\{v_i, v_{i+1}\}\cup V(G_{b_i})$, where $v_{s+1}=v_1$.
\end{claim}

\begin{proof}
  By the symmetry of the indices, it suffices to prove for the case $i=1$.
  Let $\bar{A}$ denote the set $[s+k-1]\setminus\{a_2, a_3, \dots , a_s\}$.
  Assume to the contrary that for every $j\in \bar{A}$, $G$ does not have a $v_1v_2$-path whose vertex set is $\{v_1, v_2\}\cup V(G_j)$.
  As $G_j$ is Hamilton-connected, $G_j$ satisfies one of the followings;
  \begin{enumerate}[label=(\alph*)]
    \item\label{cond:1neighbor} either $N_G(v_1)\cap V(G_j)=\emptyset$ or $N_G(v_2)\cap V(G_j)=\emptyset$, or
    \item\label{cond:0neighbor} there is a vertex $x_j$ of $G_j$ such that $N_G(v_1)\cap V(G_j)=N_G(v_2)\cap V(G_j)=\{x_j\}$.
  \end{enumerate}
  In particular, since $\{v_1,v_2\}\subseteq N_G(V(G_{a_1}))$, $G_{a_1}$ satisfies \ref{cond:0neighbor}.
  For each $j\in \bar{A}$, we define a vertex $y_j$ of $G_j$ as follows: If $G_j$ satisfies \ref{cond:1neighbor}, then let $y_j$ be an arbitrary vertex of $G_j$.
  If $G_j$ satisfies \ref{cond:0neighbor}, then let $y_j\in V(G_j)\setminus\{x_j\}$.
  Then we have $d_G(y_j)\leq n_i-1+(s-1)=n_i-2+s$ for every $j\in \bar{A}$, and in particular, we have $d_G(y_{a_1})\leq n_{a_1}-1+(s-2)=n_{a_1}-3+s$.
  Let $y_{a_2}$ be a vertex of $G_{a_2}$, and consider an independent set $I=\{y_j\mid j\in \bar{A}\cup \{a_2\}\}$.
  Since $|I|=|\bar{A}|+1=(s+k-1)-(s-1)+1=k+1$, we have 
  \begin{align*}
    \sum_{j=1}^{s+k-1}& n_j+s=n\leq \sigma_{k+1}(G)\leq \sum_{j\in\bar{A}\cup \{a_2\}} d_G(y_j)\\
    &\leq (n_{a_1}-3+s)+(n_{a_2}-1+s)+\sum_{j\in\bar{A}\setminus\{a_1\}}(n_j-2+s)
    =\sum_{j\in \bar{A}\cup\{a_2\}}n_j+(k+1)(s-2),
  \end{align*}
  and this together with Claim~\ref{cl:2factor comp mindegree} implies that
  \begin{align*}
    (k+1)(s-2)\leq \sum_{\ell=3}^{s} n_{a_\ell}=\sum_{j=1}^{s+k-1}n_j-\sum_{j\in \bar{A}\cup \{a_2\}}n_j\leq (k+1)(s-2)-s,
  \end{align*}
  a contradiction.
\end{proof}

Applying Claim~\ref{cl:2factor comp cycle reconstruct} to $C_0$ for each $i=1,2,\dots , s$ sequentially, 
we obtain a cycle $C_k=v_1u_{b_1}v_2u_{b_2}v_3\cdots v_su_{b_s}v_1$ of $H$ such that for every $i\in [s]$, there is a $v_iv_{i+1}$-path $Q_i$ of $G$ with $V(Q_i)=\{v_i, v_{i+1}\}\cup V(G_{b_i})$ ($v_{s+1}=v_1$).
Then a cycle $Q=v_1Q_1v_2Q_2v_3\cdots v_sQ_sv_1$ is a Hamilton cycle of $G[S\cup \bigcup_{i=1}^s V(G_{b_i})]$.
Combining $Q$ and a Hamilton cycle of each component $G_j$ with $j\in [s+k-1]\setminus\{b_i\mid 1\leq i\leq s\}$, we obtain a 2-factor of $G$ with exactly $k$ components.
This completes the proof of Theorem~\ref{thm:sigma0 k comp}.

\subsection{Proof of Corollary~\ref{cor:sigma0 2 comp}}

Let $G$ be a graph of order $n$ such that $\sigma_3(G)\geq n$ and every independent set $I$ of $G$ satisfies $|I|\leq \delta_G(I)-1$.
If $G$ is 1-tough, then Theorem~\ref{thm:1 tough} implies that $G$ has a dominating cycle $C$.
Since $G-V(C)$ has no edge, by applying Lemma~\ref{lem:1 vertex inclusion} to each vertex of $G-V(C)$, we obtain a Hamilton cycle of $G$.
Thus, we assume that $G$ is not 1-tough, which implies that there is a cut set $S$ of $G$ such that $\omega(G-S)\geq |S|+1=|S|+2-1$.
Then Theorem~\ref{thm:sigma0 k comp} implies that $G$ has a 2-factor with at most two cycles.
This completes the proof of Corollary~\ref{cor:sigma0 2 comp}.

\section*{Acknowledgement}

The author thanks Professor Katsuhiro Ota for giving valuable comments.
The author has been partially supported by Keio University SPRING scholarship JPMJSP2123 and JST ERATO JPMJER2301.

\end{document}